\documentclass[11pt, reqno]{amsart}
\usepackage{amsmath, amsthm, amscd, amsfonts, amssymb, graphicx, color}
\usepackage[bookmarksnumbered, colorlinks, plainpages]{hyperref}
\hypersetup{colorlinks=true,linkcolor=red, anchorcolor=green, citecolor=cyan, urlcolor=red, filecolor=magenta, pdftoolbar=true}

\newtheorem{theorem}{Theorem}[section]
\newtheorem{lemma}[theorem]{Lemma}

\theoremstyle{definition}
\newtheorem{definition}[theorem]{Definition}

\theoremstyle{remark}

\numberwithin{equation}{section}

\begin{document}

\setcounter{page}{1}

\title[Behavior of solutions to evolution inequalities]{Behavior of solutions to semilinear evolution inequalities in an annulus: the critical cases}

\author[M. B. Borikhanov, B. T. Torebek]{Meiirkhan B. Borikhanov,  Berikbol T. Torebek}

\address{{Meiirkhan B. Borikhanov \newline Khoja Akhmet Yassawi International Kazakh--Turkish University \newline Sattarkhanov ave., 29, 161200 Turkistan, Kazakhstan  \newline Institute of
Mathematics and Mathematical Modeling\newline 125 Pushkin str., 050010 Almaty, Kazakhstan}}
\email{meiirkhan.borikhanov@ayu.edu.kz, borikhanov@math.kz}

\address{{Berikbol T. Torebek \newline Institute of
Mathematics and Mathematical Modeling \newline 125 Pushkin str.,
050010 Almaty, Kazakhstan \newline Department of Mathematics: Analysis, Logic and Discrete Mathematics \newline
Ghent University, Belgium}}
\email{{torebek@math.kz, berikbol.torebek@ugent.be}}


\let\thefootnote\relax\footnote{$^{*}$Corresponding author}

\subjclass[2020]{35K70, 35A01, 35B44.}

\keywords{parabolic equations, hyperbolic equations, critical exponent, global solutions.}

\begin{abstract} In the present paper, we consider the parabolic and hyperbolic inequalities with a singular potentials and with a critical nonlinearities in the annulus domain. The problems are studied with Neumann-type and Dirichlet-type boundary conditions on the boundary. Moreover, we study the systems of problems too. We have proved that the above problems are globally unsolvable in critical cases, thereby filling the gaps the recent results by Jleli and Samet in [J. Math. Anal. Appl. 514: 2 (2022)] and in [Anal. Math. Phys. 12: 90 (2022)]. Proofs are carried out using the method of test functions with logarithmic arguments, which is being developed for the first time in bounded domains.
\end{abstract}
\maketitle
\tableofcontents

\section{Introduction}
Recently, Jleli et al. in \cite{Jleli} have studied the
nonexistence of global weak solutions to systems of semilinear
parabolic inequalities
\begin{equation}\label{SP1}\left\{\begin{array}{l}
u_t-\Delta u\geq(|x|-1)^{-\rho}|v|^{q},\,\ (t,x)\in(0,T)\times A={Q_T},\\{}\\
v_t-\Delta v\geq(|x|-1)^{-\rho}|u|^{p},\,\ (t,x)\in(0,T)\times A={Q_T} \end{array}\right.\end{equation}
and a single inequality of the form
\begin{equation}\label{P1}
u_t-\Delta u\geq(|x|-1)^{-\rho}|u|^{p},\,\ (t,x)\in(0,T)\times A={Q_T},\end{equation}
where $A=\{x\in \mathbb{R}^N: 1<|x|\leq2\}$, $N\geq2$, $\rho>0$  and $p,q>1.$

Problem \eqref{SP1} considered under two types of inhomogeneous boundary conditions:
\\$\bullet$ the Neumann type
\begin{equation}\label{SP3}
\frac{\partial u}{\partial \nu}(t,x)\geq f(x), \,\,\frac{\partial v}{\partial \nu}(t,x)\geq g(x)\,\,\,\,\text{on}\,\,\,\,\ (t,x)\in(0,T)\times \partial B_2=\Gamma,\end{equation}
\\$\bullet$ the Dirichlet type
\begin{equation}\label{SP4}
u(t,x)\geq f(x),\,\,\, v(t,x)\geq g(x)\,\,\,\,\,\text{on}\,\,\,\,\ (t,x)\in(0,T)\times \partial B_2=\Gamma,\end{equation}
where $ \partial B_2=\{x\in\mathbb{R}^N: |x|=2\},$ $\nu$ is the outward unit normal vector on $ \partial B_2$, relative to $A$, and $f, g\in L^1(\partial B_2)$ are nontrivial functions.

Namely, they have proved the following results:
\begin{itemize}
\item[(i)] Let $N\geq2,$ $p,q>1$ and let $f,g\in L^1(\partial B_2).$ If
$$\min\left\{\int_{\partial B_2}f(x)d\sigma, \int_{\partial B_2}g(x)d\sigma\right\}>0$$
and
$$\rho> 1+\min\left\{\frac{p(q+1)}{p+1},\frac{q(p+1)}{q+1}\right\},$$
then the system \eqref{SP1} with either Neumann \eqref{SP3} or Dirichlet \eqref{SP4} boundary conditions, does not admit global in time weak solutions.
\item[(ii)] If
$$0<\rho<1+\min\left\{\frac{p(q+1)}{p+1},\frac{q(p+1)}{q+1}\right\},$$
then the system \eqref{SP1} with either Neumann \eqref{SP3} or Dirichlet \eqref{SP4} boundary conditions, admits global in time weak solutions for some $f$ and $g$.\end{itemize}
Also, the similar results yields (if $f=g$ and $p=q$) for single inequality \eqref{P1}.
Hence, for the inequality \eqref{P1}, the critical exponent was defined as $p_{c}=\rho-1.$

It should be noted that in the above results the critical cases $p=p_c=\rho-1$ were not investigated and the large-time behavior of the solution in the critical cases was left open (see \cite{Jleli}, Remark 1.5).

Later on,  Jleli and Samet in \cite{Jleli1} concerned with nonexistence results for a class of nonlinear hyperbolic inequalities
\begin{equation}\label{H1}
u_{tt}-\Delta u\geq(|x|-1)^{-\rho}|u|^{p},\,\ (t,x)\in(0,T)\times A={Q_T},\end{equation} with two boundary conditions:
\\$\bullet$ the Neumann type
\begin{equation}\label{P2}
\frac{\partial u}{\partial \nu}(t,x)\geq g(x)\,\,\,\,\text{on}\,\,\,\,\ (t,x)\in(0,T)\times \partial B_2=\Gamma,\end{equation}
\\$\bullet$ the Dirichlet type
\begin{equation}\label{P3}
u(t,x)\geq
 g(x)\,\,\,\,\text{on}\,\,\,\,\ (t,x)\in(0,T)\times \partial B_2=\Gamma,\end{equation}
where $g\in L^1(\partial B_2)$.\\
They showed that:
\begin{itemize}
\item[(i)] Let $N\geq2,$ $p1$ and let $g\in L^1(\partial B_2).$ If
$$\int_{\partial B_2}g(x)d\sigma>0\,\,\,\,\,\text{and}\,\,\,\,\,p<\rho-1,$$
then the equation \eqref{H1} with either Neumann \eqref{P2} or Dirichlet \eqref{P3} boundary conditions, does not admit global in time weak solutions.
\item[(ii)] If
 $p>\rho-1$, then the equation \eqref{H1} with either Neumann \eqref{P2} or Dirichlet \eqref{P3} boundary conditions, admits global in time weak solutions for some $g>0$.\end{itemize}

However, the critical case $p=\rho-1$ was also not studied and remained open.

The  paper is devoted to studying the behavior of weak solutions in the critical case $\rho=p+1$ for the parabolic \eqref{P1} and hyperbolic \eqref{H1} inequalities with inhomogeneous Neumann \eqref{P2}
either Dirichlet \eqref{P3} boundary conditions.

Moreover, we consider weak solutions in the critical case $\rho=p+1$ to systems of semilinear parabolic inequality \eqref{SP1} and semilinear hyperbolic inequality
\begin{equation}\label{SP2}\left\{\begin{array}{l}
u_{tt}-\Delta u\geq(|x|-1)^{-\rho}|v|^{q},\,\ (t,x)\in(0,T)\times A={Q_T},\\{}\\
v_{tt}-\Delta v\geq(|x|-1)^{-\rho}|u|^{p},\,\ (t,x)\in(0,T)\times A={Q_T} \end{array}\right.\end{equation}
where $A=\{x\in \mathbb{R}^N: 1<|x|\leq2\}$, $N\geq2$, $\rho>0$  and $p,q>1.$

The system \eqref{SP1} and \eqref{SP2} will be studied under inhomogeneous Neumann \eqref{SP3} either inhomogeneous Dirichlet \eqref{SP4} boundary conditions.

In the paper \cite{Fujita1}, Fujita studied
the semilinear problem
\begin{equation}\label{0.1}\left\{\begin{array}{l}
{{u}_{t}}-\Delta u={{u}^{p}},\,\,\,(t,x)\in  (0,T)\times\mathbb{R}^N, \\{}\\
u\left( 0,x \right)=u_0\left( x \right)\ge 0,\,\,\,x\in \mathbb{R}^N, \end{array}\right.\end{equation}and proved the following results:
\begin{itemize}
\item[(a)]  If $1<p<{p}_{c}=1+\frac{2}{N}$, then the problem \eqref{0.1} admits no global positive solutions;
\item[(b)] If $p>{p}_{c},$ then for sufficiently small
initial data, the problem \eqref{0.1} admits positive global solutions.
\end{itemize}
The number $$\displaystyle{{p}_{c}=1+\frac{2}{N}}$$ is called the Fujita critical exponent,  which distinguishes between the existence and nonexistence of global in time solutions of \eqref{0.1}.

We have to mention that, when $p=p_c$, this problem was considered by Hayakawa in \cite{Hayakawa} for $N=1,2$ and Kobayashi et. al. in \cite{Kobayashi} for arbitrary $N$. For any nontrivial nonnegative initial data, it was proven that there is no nonnegative global solution.

In \cite{EsH}, Escobedo and Herrero considered systems of parabolic equations
\begin{equation}\label{SPE}\left\{\begin{array}{l}
u_t-\Delta u=|v|^{q},\,\ (t,x)\in(0,T)\times \mathbb{R}^N,\\{}\\
v_t-\Delta v=|u|^{p},\,\ (t,x)\in(0,T)\times \mathbb{R}^N, \end{array}\right.\end{equation} they have proved the following results:
\begin{itemize}
\item[(a)]  There is no global in time positive solutions if $$\frac{N}{2}\leq \max\left\{\frac{p+1}{pq-1},\frac{q+1}{pq-1}\right\};$$
\item[(b)] There exist global in time positive solutions if $$\frac{N}{2}> \max\left\{\frac{p+1}{pq-1},\frac{q+1}{pq-1}\right\}.$$
\end{itemize}

In \cite{Bandle}, Bandle and Levine considered the problem \eqref{0.1} on the exterior domain of $\mathbb{R}^N$ with homogeneous Dirichlet boundary conditions, and they proved that the critical exponent is still $\displaystyle{{p}_{c}=1+\frac{2}{N}}.$

Moreover, the first contribution on the subject of the nonexistence of wave inequality in the
whole space is \cite{Kato} and more generally \cite{Mitidieri1}, where it was shown that the hyperbolic inequality
\begin{equation}\label{KAT}
u_{tt}-\Delta u \geq |u|^p\,\,\,\, \text{in}\,\,\, (0,\infty) \times \mathbb{R}^N
\end{equation}admits another critical exponent (known as Kato exponent) $\displaystyle\large p_K=\frac{N+1}{N-1}$. In \cite{Pohozaev},
Pohozaev and Veron generalized the work in \cite{Mitidieri1} and pointed out the sharpness of
$p_K$ for the problem \eqref{KAT}. Namely, they proved that,
\begin{itemize}
\item[(a)]  if $N\geq 2,\,\int_{\mathbb{R}^N}\partial_tu(0,x)>0$ and $1<p\leq p_K$, then problem \eqref{KAT} admits
no global weak solution;
\item[(b)] if $p>p_K$, then problem \eqref{KAT} admits global positive solutions that satisfy
property $\int_{\mathbb{R}^N}\partial_tu(0,x)>0$.
\end{itemize}

The critical exponents have been studied by a large number of authors in the whole space (for example(\cite{Borikhanov0, Cazenave, Kirane0, DelSantiago, Deng, Levine, Mitidieri, Kirane1, Weissler}), cone-like domains (for example \cite{Cast, Igar, Laptev, Levine1}) and exterior domain (for example \cite{Alqahtani,  Borikhanov, Dabb, Ikeda, Ikehata, Jleli3, Jleli0, Levine2, Zhang}),  for more details about critical exponents and blowup theorems one can see in \cite{Levine, Mitidieri, Soup, Samarskii} and the references therein.

\section{Main results}
In this section, we derive the main results of this work.
\begin{definition}[Weak solution]\label{WS1} We say that $u\in L^p_\text{loc}({Q_T})$ is a local weak solution to \eqref{P1}-\eqref{P2}, if the following inequality
\begin{equation}\label{W1}\begin{split}
&\int_{Q_T}(|x|-1)^{-\rho}|u|^{p}\varphi  dx dt +\int_\Gamma f\varphi d\sigma dt\leq -\int_{Q_T} u\varphi_tdx dt-\int_{Q_T} u\Delta\varphi dx dt,
\end{split}\end{equation}
holds for all $\varphi\in C^{2}_c({Q_T}), \varphi\geq0, \varphi(T,\cdot)=0$ with $\displaystyle\frac{\partial \varphi}{\partial \nu}|_\Gamma=0$ and the notation $d\sigma$ is the surface measure on $\partial{\Omega}$.
\end{definition}

\begin{definition}[Weak solution]\label{WS2}A locally integrable function $u\in L^p_\text{loc}({Q_T})$ is called a local  weak solution of \eqref{P1}-\eqref{P3}, if
\begin{equation}\label{W2}\begin{split}
&\int_{Q_T}(|x|-1)^{-\rho}|u|^{p}\psi  dx dt -\int_\Gamma f\frac{\partial \psi}{\partial \nu}d\sigma dt\leq -\int_{Q_T} u\psi_tdx dt-\int_{Q_T} u\Delta\psi dx dt,
\end{split}\end{equation}
holds true for any function $\psi\in C^{2}_c({Q_T}), \psi\geq0, \psi(T,\cdot)=0$ and $\psi|_\Gamma=0,\,\displaystyle\frac{\partial \psi}{\partial \nu}|_\Gamma\leq0$.
\end{definition}

\begin{definition}[Weak solution]\label{WS3} We say that $u\in L^p_\text{loc}({Q_T})$ is a weak solution to \eqref{H1}-\eqref{P2}, if the inequality
\begin{equation}\label{W3}\begin{split}
&\int_{Q_T}(|x|-1)^{-\rho}|u|^{p}\varphi  dx dt +\int_\Gamma f\varphi d\sigma dt\leq \int_{Q_T} u\varphi_{tt}dx dt-\int_{Q_T} u\Delta\varphi dx dt,
\end{split}\end{equation}
holds for all $\varphi\in C^{2}_c({Q_T}), \varphi\geq0, \varphi(T,\cdot)=0$ and $\displaystyle\frac{\partial \varphi}{\partial \nu}|_\Gamma=0$.
\end{definition}

\begin{definition}[Weak solution]\label{WS4} A local weak solution of the problem  \eqref{H1}-\eqref{P3} is a function  $u\in L^p_\text{loc}({Q_T})$ such that
\begin{equation}\label{W4}\begin{split}
&\int_{Q_T}(|x|-1)^{-\rho}|u|^{p}\psi  dx dt -\int_\Gamma f\frac{\partial \psi}{\partial \nu}d\sigma dt\leq \int_{Q_T} u\psi_{tt}dxdt-\int_{Q_T} u\Delta\psi dx dt,
\end{split}\end{equation}
for any $\psi\in C^{2}_c({Q_T}), \psi\geq0, \psi(T,\cdot)=0$ and $\psi|_\Gamma=0,\,\displaystyle\frac{\partial \psi}{\partial \nu}|_\Gamma\leq0$.
\end{definition}

\begin{definition}[Weak solution]\label{WS5} We say that a pair of functions $(u,v)\in L^p_\text{loc}({Q_T})\times L^q_\text{loc}({Q_T})$ is a local weak solution of the problem  \eqref{SP1}-\eqref{SP3}  if
\begin{equation}\label{W5}\begin{split}
&\int_{Q_T}(|x|-1)^{-\rho}|u|^{p}\varphi  dx dt+\int_\Gamma g\varphi d\sigma dt\leq -\int_{Q_T} v\varphi_{t}dxdt-\int_{Q_T} v\Delta\varphi dx dt,
\end{split}\end{equation}
\begin{equation}\label{W6}\begin{split}
&\int_{Q_T}(|x|-1)^{-\rho}|v|^{p}\varphi dx dt+\int_\Gamma f\varphi d\sigma dt\leq -\int_{Q_T} u\varphi_{t}dxdt-\int_{Q_T} u\Delta\varphi dx dt,
\end{split}\end{equation}
for any $\varphi\in C^{2}_c({Q_T}), \varphi\geq0, \varphi(T,\cdot)=0$ and $\displaystyle\frac{\partial \varphi}{\partial \nu}|_\Gamma=0$.
\end{definition}

\begin{definition}[Weak solution]\label{WS6} We say that a pair of functions $(u,v)\in L^p_\text{loc}({Q_T})\times L^q_\text{loc}({Q_T})$ is a local weak solution of the problem  \eqref{SP1}-\eqref{SP4}  if
\begin{equation}\label{W7}\begin{split}
&\int_{Q_T}(|x|-1)^{-\rho}|u|^{p}\psi  dx dt-\int_\Gamma g\frac{\partial \psi}{\partial \nu} d\sigma dt\leq -\int_{Q_T} v\psi_{t}dxdt-\int_{Q_T} v\Delta\psi dx dt,
\end{split}\end{equation}
\begin{equation}\label{W8}\begin{split}
&\int_{Q_T}(|x|-1)^{-\rho}|v|^{p}\psi dx dt-\int_\Gamma f\frac{\partial \psi}{\partial \nu} d\sigma dt\leq -\int_{Q_T} u\psi_{t}dxdt-\int_{Q_T} u\Delta\psi dx dt,
\end{split}\end{equation}
for any $\psi\in C^{2}_c({Q_T}), \psi\geq0, \psi(T,\cdot)=0$ and $\psi|_\Gamma=0,\,\displaystyle\frac{\partial \psi}{\partial \nu}|_\Gamma\leq0$.
\end{definition}

\begin{definition}[Weak solution]\label{WS7} We say that a pair of functions $(u,v)\in L^p_\text{loc}({Q_T})\times L^q_\text{loc}({Q_T})$ is a local weak solution of the problem  \eqref{SP2}-\eqref{SP3}  if
\begin{equation}\label{W9}\begin{split}
&\int_{Q_T}(|x|-1)^{-\rho}|u|^{p}\varphi  dx dt+\int_\Gamma g\varphi d\sigma dt\leq \int_{Q_T} v\varphi_{tt}dxdt-\int_{Q_T} v\Delta\varphi dx dt,
\end{split}\end{equation}
\begin{equation}\label{W10}\begin{split}
&\int_{Q_T}(|x|-1)^{-\rho}|v|^{p}\varphi dx dt+\int_\Gamma f\varphi d\sigma dt\leq \int_{Q_T} u\varphi_{tt}dxdt-\int_{Q_T} u\Delta\varphi dx dt,
\end{split}\end{equation}
for any $\varphi\in C^{2}_c({Q_T}), \varphi\geq0, \varphi(T,\cdot)=0$ and $\displaystyle\frac{\partial \varphi}{\partial \nu}|_\Gamma=0$.
\end{definition}

\begin{definition}[Weak solution]\label{WS8} We say that a pair of functions $(u,v)\in L^p_\text{loc}({Q_T})\times L^q_\text{loc}({Q_T})$ is a local weak solution of the problem  \eqref{SP2}-\eqref{SP4}  if
\begin{equation}\label{W11}\begin{split}
&\int_{Q_T}(|x|-1)^{-\rho}|u|^{p}\psi  dx dt-\int_\Gamma g\frac{\partial\psi}{\partial\nu} d\sigma dt\leq \int_{Q_T} v\psi_{tt}dxdt-\int_{Q_T} v\Delta\psi dx dt,
\end{split}\end{equation}
\begin{equation}\label{W12}\begin{split}
&\int_{Q_T}(|x|-1)^{-\rho}|v|^{p}\psi  dx dt-\int_\Gamma f\frac{\partial\psi}{\partial\nu}d\sigma dt\leq \int_{Q_T} u\psi_{tt}dxdt-\int_{Q_T} u\Delta\psi dx dt,
\end{split}\end{equation}
for any $\psi\in C^{2}_c({Q_T}), \psi\geq0, \psi(T,\cdot)=0$ and $\psi|_\Gamma=0,\,\displaystyle\frac{\partial \psi}{\partial \nu}|_\Gamma\leq0$.
\\If $T=+\infty,$ then $u$ is called a global in time weak solution.
\end{definition}

\begin{theorem}\label{TT1} Let $N\geq2, p>1$ and $f\in L^1(\partial B_2).$ If
$$\int_{\partial B_2}f(x)d\sigma>0\,\,\,\,\,\text{and}\,\,\,\,\,p=p_c=\rho-1,$$
then
\begin{itemize}
\item[\textbf{(i).}]
the problem \eqref{P1} with either Neumann \eqref{P2} or Dirichlet \eqref{P3} boundary conditions, does not admit global weak solutions.
\item[\textbf{(ii).}] the problem \eqref{H1} with either Neumann \eqref{P2} or Dirichlet \eqref{P3} boundary conditions, admits no global in time weak solutions.
\end{itemize}
\end{theorem}

\begin{theorem}\label{TT2} Let $N\geq2$ and $f,g\in L^1(\partial B_2).$ Suppose that $(u,v)\in L^p_{\text{loc}}({Q_T})\times L^q_{\text{loc}}({Q_T})$ is a global weak solution to problem \eqref{SP1}-\eqref{SP3}. If
\begin{equation}\label{ss}
\int_{\partial B_2}f(x)d\sigma>0,\,\, \int_{\partial B_2}g(x)d\sigma>0
\end{equation}and
\begin{equation}\label{ssT}
\rho=\left\{\begin{array}{l}
1+\frac{p(q+1)}{p+1}\,\,\,\,\,\,\,\,\,\text{if}\,\,\,\,p\geq q,\\
1+\frac{q(p+1)}{q+1}\,\,\,\,\,\,\,\,\,\text{if}\,\,\,\,p<q,\end{array}\right.
\end{equation}then
\begin{itemize}
\item[\textbf{(i).}]
the problem \eqref{SP1} with Neumann \eqref{SP3} or Dirichlet \eqref{SP4} boundary conditions admits no global in time weak solutions.
\item[\textbf{(ii).}] the problem \eqref{SP2} with either Neumann \eqref{SP3} or Dirichlet \eqref{SP4} boundary conditions, does not admit global weak solutions.

\end{itemize}
\end{theorem}

\subsection{Test functions}This section will cover some test functions and their properties. Additionally, we will establish some useful estimates related to the test functions.

Now, we consider the test function for the coefficients $T,R$ is sufficiently large in the form
\begin{equation}\label{T0}
\varphi(t,x)=\mu(t)\xi(x),\,\,\,t\in(0,T),\,x\in A,\end{equation}where
\begin{equation}\label{T1}
 \mu(t)=\left(1-\frac{t}{T}\right)^l,\,\,\,t>0,\,\, l>\frac{p+1}{p-1}
\end{equation}
and let $\xi(x)$ be a family of smooth functions in $A$ satisfying
\begin{equation}\label{T2}
 0\leq\xi(x)\leq1,\,\,\,\text{supp}(\xi)\subset\subset 1+\frac{1}{R}<|x|\leq2,\,\xi=1\,\,\,\text{in}\,\,\,1+\frac{1}{\sqrt{R}}<|x|\leq2.
\end{equation}

\begin{lemma}\label{critcal case}Let the test function $\xi(x)$ be introduced in the form
\begin{equation}\label{CRC1} \xi(x)=\Phi^k\left(\frac{\ln(R(|x|-1))}{\ln R}\right),\,\, k>2,\end{equation}where $\Phi:\mathbb{R}\to[0,1]$ be a smooth function function such that
\begin{equation*}
\Phi(z)=
 \begin{cases}
   0 &\text{if\,\,\,\, $-\infty<z\leq0 $},\\
   \nearrow &\text{if\,\,\,\, $0<z<\frac{1}{2}$},\\
   1 &\text{if\,\,\,\, $\frac{1}{2}\leq z\leq 1$}.
 \end{cases}\end{equation*}
 Consequently,
 \begin{equation*}
\Phi\left(\frac{\ln(R(|x|-1))}{\ln R}\right)=
 \begin{cases}
   0 &\text{if\,\,\,\, $1<|x|<1+\frac{1}{R}$},\\
   \nearrow &\text{if\,\,\,\, $1+\frac{1}{R}<|x|<1+\frac{1}{\sqrt{R}}$},\\
   1 &\text{if\,\,\,\, $1+\frac{1}{\sqrt{R}}<|x|\leq2$}.
 \end{cases}\end{equation*}
Suppose that
 \begin{equation}\label{TPQ}
  |\Phi'(z)|\leq C,\,\,\,|\Phi''(z)|\leq C,
 \end{equation}
 then, the next estimate holds true
 \begin{equation}\label{TPQ1}\begin{split}
 |\Delta\xi(x)|&\leq  \frac{C}{(|x|-1)^2\ln^2 R }\Phi^{k-2}\left(\frac{\ln(R(|x|-1))}{\ln R}\right)
\\&+\frac{C}{(|x|-1)^2\ln R }\Phi^{k-1}\left(\frac{\ln(R(|x|-1))}{\ln R}\right). \end{split}\end{equation} \end{lemma}
\begin{proof}[Proof of Lemma \ref{critcal case}.] Indeed, the function $\Phi(x)$ is a radial, then we arrive at

\begin{equation*}\label{A0008}\begin{split}
\Delta\xi(r)&=\left(\frac{d^2}{dr^2}+\frac{N-1}{r}\frac{d}{dr}\right)\Phi^k\left(\frac{\ln(R(r-1))}{\ln R}\right)
\\&=k(k-1)\Phi^{k-2}\left(\frac{\ln(R(r-1))}{\ln R}\right)\left[\Phi'\left(\frac{\ln(R(r-1))}{\ln R}\right)\right]^2\frac{1}{\ln^2 R (r-1)^2}
\\&+k\Phi^{k-1}\left(\frac{\ln(R(r-1))}{\ln R}\right)\Phi''\left(\frac{\ln(R(r-1))}{\ln R}\right)\frac{1}{\ln^2 R (r-1)^2}
\\&+k\Phi^{k-1}\left(\frac{\ln(R(r-1))}{\ln R}\right)\Phi'\left(\frac{\ln(R(r-1))}{\ln R}\right)\frac{N-2}{\ln R (r-1)^2},
\end{split}\end{equation*}where $r=|x|=(x_1^2+x_2^2+...+x_n^2)^\frac{1}{2}$ and $C>0$ is an arbitrary constant.
\\Moreover, from  \eqref{TPQ}, we obtain
\begin{equation*}\label{A0}\begin{split}
|\Delta\xi(r)|&\leq \frac{C}{\ln^2 R (r-1)^2}\Phi^{k-2}\left(\frac{\ln(R(r-1))}{\ln R}\right)
\\&+\frac{C}{\ln R (r-1)^2}\Phi^{k-1}\left(\frac{\ln(R(r-1))}{\ln R}\right),
\end{split}\end{equation*}
which completes the proof.\end{proof}
\begin{lemma}\label{POP} Let $\rho>0,\,p>1$ and
\begin{equation}\label{POP0}
\varphi=\mu(t)\xi(x),
\end{equation}  where the functions $\mu(t), \xi(x)$ are defined by \eqref{T1} and \eqref{CRC1}, respectively.  \\For sufficiently large $T, R$ there holds
\begin{equation}\label{POP1}\begin{split}
&\mathcal{I}_1(\varphi,p)=\int_{Q_T} (|x|-1)^\frac{\rho}{p-1}\varphi^{-\frac{1}{p-1}}|\varphi_t|^\frac{p}{p-1}dx dt \leq CT^{1-\frac{p}{p-1}},
\end{split}\end{equation}
\begin{equation}\label{POP2}\begin{split}
&\mathcal{I}_2(\varphi,p)=\int_{Q_T} (|x|-1)^\frac{\rho}{p-1}\varphi^{-\frac{1}{p-1}}|\varphi_{tt}|^\frac{p}{p-1}dx dt \leq CT^{1-\frac{2p}{p-1}}
\end{split}\end{equation}
and
\begin{equation}\label{POP3}\begin{split}
\mathcal{I}_3(\varphi,p)&=\int_{Q_T} (|x|-1)^\frac{\rho}{p-1}\varphi^{-\frac{1}{p-1}}|\Delta\varphi|^\frac{p}{p-1}dx dt\\&\leq CT\left[\left(\ln R\right)^{-\frac{2p}{p-1}}+\left(\ln R\right)^{-\frac{p}{p-1}}\right]R^{-\frac{1}{2}\left(\frac{\rho-p-1}{p-1}\right)}.
\end{split}\end{equation}
\end{lemma}
\begin{proof}[Proof of Lemma \ref{POP}.] In view of \eqref{POP0} it follows that
\begin{equation*}\begin{split}
\mathcal{I}_1(\varphi,p)&=\left(\int_0^T \mu^{-\frac{1}{p-1}}|\mu'|^\frac{p}{p-1} dt\right) \left( \int\limits_{1+\frac{1}{R}<|x|< 2} (|x|-1)^\frac{\rho}{p-1}\xi dx \right).
\end{split}\end{equation*}
Hence, by \eqref{T2} we obtain
\begin{equation}\label{POP4}\begin{split}
 \int\limits_{1+\frac{1}{R}<|x|< 2} (|x|-1)^\frac{\rho}{p-1}\xi dx&\leq C \int\limits_{1+\frac{1}{R}<|x|< 2} (|x|-1)^\frac{\rho}{p-1} dx
\\&\leq  C.
\end{split}\end{equation}
Moreover, taking into account a simple calculation gives
\begin{equation}\label{LP2}\begin{split}
\int_0^T \mu^{-\frac{1}{p-1}}\left|\mu'\right|^\frac{p}{p-1} dt&=\int_0^T \left(1-\frac{t}{T}\right)^{-\frac{l}{p-1}}\left|lT^{-1}\left(1-\frac{t}{T}\right)^{l-1}\right|^\frac{p}{p-1} dt
\\&=CT^{1-\frac{p}{p-1}},
\end{split}\end{equation}which gives the desired result.
\\Consequently, using \eqref{POP4} and the expression
\begin{equation}\label{OPP}\begin{split}
\int_0^T\mu^{-\frac{1}{p-1}}|\mu''|^\frac{p}{p-1} dt&=\int_0^T \left(1-\frac{t}{T}\right)^{-\frac{l}{p-1}}\left|l(l-1)T^{-2}\left(1-\frac{t}{T}\right)^{l-2}\right|^\frac{p}{p-1} dt
\\&=CT^{1-\frac{2p}{p-1}},
\end{split}\end{equation}we obtain \eqref{POP2}.
\\Similarly, the integral $\mathcal{I}_3(\varphi,p)$ can be rewritten as
\begin{equation*}\label{POP7}\begin{split}
&\mathcal{I}_3(\varphi,p)=\left(\int_0^T \mu^{-\frac{1}{p-1}}|\mu|^\frac{p}{p-1} dt\right) \left( \int\limits_{1+\frac{1}{R}<|x|< 1+\frac{1}{\sqrt{R}}} (|x|-1)^\frac{\rho}{p-1}\xi^{-\frac{1}{p-1}}|\Delta\xi|^\frac{p}{p-1} dx \right).
\end{split}\end{equation*}
At this stage, from \eqref{TPQ1} and the following inequality
\begin{equation}\label{inequality}
(a+b)^m\leq 2^{m-1}(a^m+b^m),\,\,\,a\geq0,\,b\geq0,\,m=\frac{p}{p-1},\end{equation} we arrive at
\begin{equation*}\label{A1}\begin{split}
\int\limits_{1+\frac{1}{R}<|x|< 1+\frac{1}{\sqrt{R}}} &(|x|-1)^\frac{\rho}{p-1}\xi^{-\frac{1}{p-1}}|\Delta\xi|^\frac{p}{p-1} dx
\\&\leq C\int\limits_{1+\frac{1}{R}<|x|< 1+\frac{1}{\sqrt{R}}}  (|x|-1)^\frac{\rho}{p-1}\xi^{-\frac{1}{p-1}}\left[\frac{1}{(|x|-1)^2\ln^2 R }\left|\xi^{\frac{k-2}{k}}\right|\right]^\frac{p}{p-1}   dx\\& +C\int\limits_{1+\frac{1}{R}<|x|< 1+\frac{1}{\sqrt{R}}}  (|x|-1)^\frac{\rho}{p-1}\xi^{-\frac{1}{p-1}}\left[\frac{1}{(|x|-1)^2\ln R}\left|\xi^{\frac{k-1}{k}}\right|\right]^\frac{p}{p-1}   dx.
\end{split}\end{equation*}
Using \eqref{T2} and \eqref{CRC1} with $k>\frac{2p}{p-1}$ it follows that
\begin{equation}\label{A2}\begin{split}
\int\limits_{1+\frac{1}{R}<|x|< 1+\frac{1}{\sqrt{R}}} (|x|-1)^\frac{\rho}{p-1}&\left[\frac{1}{(|x|-1)^2\ln^2 R }\right]^\frac{p}{p-1}dx
\\&=\left[\frac{1}{\ln^2R}\right]^\frac{p}{p-1}\int\limits_{1+\frac{1}{R}<|x|< 1+\frac{1}{\sqrt{R}}}(|x|-1)^{\frac{\rho}{p-1}-\frac{2p}{p-1}}dx
\\&\stackrel{|x|=r}{=}\left[\frac{1}{\ln^2R}\right]^\frac{p}{p-1}\int\limits_{1+\frac{1}{R}<r< 1+\frac{1}{\sqrt{R}}} (r-1)^{\frac{\rho}{p-1}-\frac{2p}{p-1}} r^{N-1}dr
\\&\stackrel{r-1=s}{\leq} C\left[\frac{1}{\ln^2R}\right]^\frac{p}{p-1} \int\limits_{\frac{1}{R}<s< \frac{1}{\sqrt{R}}} s^{\frac{\rho}{p-1}-\frac{2p}{p-1}} ds
\\&\leq \left(\ln R\right)^{-\frac{2p}{p-1}} \left(R^{-\frac{1}{2}\left(\frac{\rho-p-1}{p-1}\right)}-R^{-\frac{\rho-p-1}{p-1}}\right)
\\&\leq\left(\ln R\right)^{-\frac{2p}{p-1}} R^{-\frac{1}{2}\left(\frac{\rho-p-1}{p-1}\right)}.
\end{split}\end{equation}
Similarly, one obtains
\begin{equation}\label{A3}\begin{split}
\int\limits_{1+\frac{1}{R}<|x|< 1+\frac{1}{\sqrt{R}}} (|x|-1)^\frac{\rho}{p-1}&\left[\frac{1}{(|x|-1)^2\ln R }\right]^\frac{p}{p-1}dx\leq\left(\ln R\right)^{-\frac{p}{p-1}} R^{-\frac{1}{2}\left(\frac{\rho-p-1}{p-1}\right)}.
\end{split}\end{equation}
It is also easy to verify by a direct calculation that
\begin{equation}\label{LP10}\begin{split}
\int_0^T\mu^{-\frac{1}{p-1}}|\mu|^\frac{p}{p-1} dt&=\int_0^T\left(1-\frac{t}{T}\right)^{-\frac{l}{p-1}}\left|\left(1-\frac{t}{T}\right)^{l}\right|^{\frac{p}{p-1}}dt
\\&=CT.\end{split}\end{equation}
Combining \eqref{A2}, \eqref{A3}  and \eqref{LP10} we obtain \eqref{POP3}.\end{proof}
\begin{lemma}\label{LL2} Let the function $\varphi(t,x)$ be as in \eqref{T0} and  $f\in L^1(\partial B_2)$, then there holds
\begin{equation*}\label{MP1}\begin{split}
\int_\Gamma f\varphi d\sigma dt \leq CT\int_{\partial B_2}f(x)d\sigma,\end{split}\end{equation*}
where $C>0$ independent of $T$.
\end{lemma}
\begin{proof}
From the properties of the test function \eqref{T0}-\eqref{T2}, it yields that

\begin{equation*}\label{TF10}\begin{split}
\int_\Gamma f\varphi d\sigma dt =\int_\Gamma f(x)\mu(t)\xi(x) d\sigma dt&=\left(\int_0^T\left(1-\frac{t}{T}\right)^ldt\right)\left(\int_{\partial B_2}f(x)\xi(x)d\sigma\right)
\\&\leq CT\left(\int_{\partial B_2}f(x)d\sigma\right),
\end{split}\end{equation*}
which completes the proof.
\end{proof}
Next, we introduce the function $\psi(t,x)$ for sufficiently large $T,R$ such that
\begin{equation}\label{S0}
\psi(t,x)=\psi_1(t)\psi_2(x)=\mu(t)H(x)\xi(x),\,\,\,t\in(0,T),\,x\in A,\end{equation}where the functions $\mu(t), \xi(x)$ are given as in \eqref{T1} and \eqref{T2}, respectively.
\\Consequently, the function $H$ defined in $A$ by
\begin{equation}\label{HL1}
H(x)=\left\{\begin{array}{l}
\ln2-\ln|x|,\,\,\,\,\,\,\,\,\,\,\,\,\,\,\,\text{if}\,\,\,\,N=2,\\
2^{N-2}|x|^{2-N}-1,\,\,\,\text{if}\,\,\,\,N\geq3.\end{array}\right.
\end{equation} In addition, $H$ is nonnegative, harmonic and satisfies the condition $H|_{\partial B_2}=0.$

\begin{lemma}\label{LS90}For the function $\psi(t,x)$ on $\Gamma$ the following inequality
\begin{equation*}\label{S1}
\frac{\partial\psi}{\partial\nu}(t,x)=-C(N)\mu(t)\leq0\end{equation*}
holds true, where
\begin{equation}\label{CNN}
C(N)=\left\{\begin{array}{l}
2^{-1},\,\,\,\,\,\,\,\,\,\,\,\,\,\,\,\,\,\,\,\,\,\,\,\,\text{if}\,\,\,\,N=2,\\
2^{-1}(N-2),\,\,\,\text{if}\,\,\,\,N\geq3. \end{array}\right.\end{equation}\end{lemma}

\begin{proof}[Proof of Lemma \ref{LS90}.] Using the fact that $\nabla\xi(x)=0$ with $\xi(x)=1$ on $ 1+\frac{1}{\sqrt{R}}<|x|\leq2$ from \eqref{T2}, we arrive at
\begin{equation*}\label{S2}\begin{split}
\nabla\psi_2(x)&=\nabla\left(H(x)\xi(x)\right)
\\&=H(x)\nabla\xi(x)+\xi(x)\nabla H(x)
\\&=\nabla H(x). \end{split}\end{equation*}
Due to \eqref{HL1} for $x\in\partial B_2$, it yields
\begin{equation*}\label{S3}
\frac{\partial\psi_2}{\partial\nu}(x)=-C(N),\end{equation*}
where
\begin{equation*}\label{CN}
C(N)=\left\{\begin{array}{l}
2^{-1},\,\,\,\,\,\,\,\,\,\,\,\,\,\,\,\,\,\,\,\,\,\,\,\,\text{if}\,\,\,\,N=2,\\
2^{-1}(N-2),\,\,\,\text{if}\,\,\,\,N\geq3. \end{array}\right.\end{equation*}
Finally, in view of \eqref{S0} we conclude that
\begin{equation*}\label{S4}
\frac{\partial\psi}{\partial\nu}(t,x)=-C(N)\mu(t)=\left\{\begin{array}{l}
2^{-1}\mu(t)\leq0,\,\,\,\,\,\,\,\,\,\,\,\,\,\,\,\,\,\,\,\,\,\,\,\,\text{if}\,\,\,\,N=2,\\
2^{-1}(N-2)\mu(t)\leq0,\,\,\,\text{if}\,\,\,\,N\geq3\end{array}\right.\end{equation*}
holds for all $(t,x)\in \Gamma$.\end{proof}
\begin{lemma}\label{DWQ} Let $\rho>0,\,p>1$ and $\psi(t,x)$
as in \eqref{S0}.  For sufficiently large $T, R$ we have
\begin{equation}\label{DWQ1}\begin{split}
&\mathcal{J}_1(\psi,p)=\int_{Q_T} (|x|-1)^\frac{\rho}{p-1}\psi^{-\frac{1}{p-1}}|\psi_t|^\frac{p}{p-1}dx dt \leq CT^{1-\frac{p}{p-1}},
\end{split}\end{equation}
\begin{equation}\label{DWQ2}\begin{split}
&\mathcal{J}_2(\psi,p)=\int_{Q_T} (|x|-1)^\frac{\rho}{p-1}\psi^{-\frac{1}{p-1}}|\psi_{tt}|^\frac{p}{p-1}dx dt \leq CT^{1-\frac{2p}{p-1}}
\end{split}\end{equation}
and
\begin{equation}\label{DWQ3}\begin{split}
\mathcal{J}_3(\psi,p)&=\int_{Q_T} (|x|-1)^\frac{\rho}{p-1}\psi^{-\frac{1}{p-1}}|\Delta\psi|^\frac{p}{p-1}dx dt
\\&\leq CT\left[\left(\ln R\right)^{-\frac{2p}{p-1}}R^{-\frac{1}{2}\left(\frac{\rho-p-1}{p-1}\right)}+\left(\ln R\right)^{-\frac{p}{p-1}} R^{-\frac{1}{2}\left(\frac{\rho-1}{p-1}\right)}\right].
\end{split}\end{equation}
\end{lemma}
\begin{proof}[Proof of Lemma \ref{DWQ}.] Taking into account \eqref{S0}, we obtain
\begin{equation*}\begin{split}
\mathcal{J}_1(\psi,p)&=\left(\int_0^T \mu^{-\frac{1}{p-1}}|\mu'|^\frac{p}{p-1} dt\right) \left( \int\limits_{1+\frac{1}{R}<|x|< 2} (|x|-1)^\frac{\rho}{p-1}H(x)\xi(x) dx \right).
\end{split}\end{equation*}
On other hand, by \eqref{HL1} it can be seen that
\begin{equation}\label{HLL}
 H(x)\leq C\,\,\,\text{for}\,\,\, 1+\frac{1}{R}<|x|\leq2.
\end{equation}
Consequently, due to \eqref{POP4} and \eqref{LP2} we get \eqref{DWQ1}.
\\Moreover, using the estimates \eqref{POP4}, \eqref{OPP} and noting \eqref{HLL} we deduce that \eqref{DWQ2}.
\\Next, the integral $\mathcal{J}_3(\psi,p)$ can be rewritten as
\begin{equation*}\label{PDWQ7}\begin{split}
&\mathcal{I}_3(\psi,p)=\left(\int_0^T \mu^{-\frac{1}{p-1}}|\mu|^\frac{p}{p-1} dt\right) \left( \int\limits_{1+\frac{1}{R}<|x|< 1+\frac{1}{\sqrt{R}}} (|x|-1)^\frac{\rho}{p-1}(H(x)\xi(x))^{-\frac{1}{p-1}}|\Delta(H(x)\xi(x))|^\frac{p}{p-1} dx \right).
\end{split}\end{equation*}
Moreover, by \eqref{HL1}, we get
\begin{equation}\label{GH}
H^{-\frac{1}{p-1}}(x)\leq C, |H(x)|\leq C, |\nabla H(x)|\leq C,\,\,\,1+\frac{1}{R}<x<1+\frac{1}{\sqrt{R}}.
\end{equation}
Therefore, taking into account that the function $H(x)$ is harmonic with \eqref{GH} and the remaining \eqref{TPQ1}, we arrive at
\begin{equation*}\begin{split}
|\Delta\psi_2(x)|=|\Delta\left[H(x)\xi(x)\right]|&\leq H(x)|\Delta\left[\xi(x)\right]|+2|\nabla H(x)||\nabla\xi(x)|
\\&\leq \frac{CH(x)}{ (|x|-1)^2\ln^2 R}\Phi^{k-2}\left(\frac{\ln(R(|x|-1))}{\ln R}\right)
\\&+\frac{CH(x)}{(|x|-1)^2\ln R }\Phi^{k-1}\left(\frac{\ln(R(|x|-1))}{\ln R}\right)
\\&+\frac{C}{(|x|-1)\ln R}\Phi^{k-1}\left(\frac{\ln(R(|x|-1))}{\ln R}\right)
\\&\leq  C\left(\frac{1}{(|x|-1)^2\ln^2 R }\xi^{\frac{k-2}{k}}\left(x\right)+\frac{1}{(|x|-1)\ln R}\xi^{\frac{k-1}{k}}\left(x\right)\right).\end{split}\end{equation*}
Using the inequality \eqref{inequality}, there holds
\begin{equation*}\begin{split}
\int\limits_{1+\frac{1}{R}<|x|< 1+\frac{1}{\sqrt{R}}} &(|x|-1)^\frac{\rho}{p-1}\xi^{-\frac{1}{p-1}}|\Delta\xi|^\frac{p}{p-1} dx
\\&\leq C\int\limits_{1+\frac{1}{R}<|x|< 1+\frac{1}{\sqrt{R}}}  (|x|-1)^\frac{\rho}{p-1}\xi^{-\frac{1}{p-1}}\left[\frac{1}{(|x|-1)^2\ln^2 R }\left|\xi^{\frac{k-2}{k}}\right|\right]^\frac{p}{p-1}   dx\\& +C\int\limits_{1+\frac{1}{R}<|x|< 1+\frac{1}{\sqrt{R}}}  (|x|-1)^\frac{\rho}{p-1}\xi^{-\frac{1}{p-1}}\left[\frac{1}{(|x|-1)\ln R}\left|\xi^{\frac{k-1}{k}}\right|\right]^\frac{p}{p-1}   dx.\end{split}\end{equation*}
Then, noting that $k>\frac{2p}{p-1}$ and \eqref{T2} we arrive at
\begin{equation*}\begin{split}
\int\limits_{1+\frac{1}{R}<|x|< 1+\frac{1}{\sqrt{R}}} (|x|-1)^\frac{\rho}{p-1}&\xi^{-\frac{1}{p-1}}|\Delta\xi|^\frac{p}{p-1} dx
\\&\leq C\int\limits_{1+\frac{1}{R}<|x|< 1+\frac{1}{\sqrt{R}}} (|x|-1)^\frac{\rho}{p-1}\left[\frac{1}{(|x|-1)^2\ln^2 R }\right]^\frac{p}{p-1}dx
\\&+C\int\limits_{1+\frac{1}{R}<|x|< 1+\frac{1}{\sqrt{R}}}(|x|-1)^\frac{\rho}{p-1}\left[\frac{1}{(|x|-1)\ln R}\right]^\frac{p}{p-1}dx.\end{split}\end{equation*}
At this stage, we calculate the last integral in the following form 
\begin{equation}\label{DWQ7}\begin{split}
\int\limits_{1+\frac{1}{R}<|x|< 1+\frac{1}{\sqrt{R}}}(|x|-1)^\frac{\rho}{p-1}&\left[\frac{1}{(|x|-1)\ln R}\right]^\frac{p}{p-1}dx
\\&=\left[\frac{1}{\ln R}\right]^\frac{p}{p-1}\int\limits_{1+\frac{1}{R}<|x|< 1+\frac{1}{\sqrt{R}}}(|x|-1)^{\frac{\rho}{p-1}-\frac{p}{p-1}}dx
\\&\stackrel{|x|=r}{=}\left[\frac{1}{\ln R}\right]^\frac{p}{p-1}\int\limits_{1+\frac{1}{R}<r< 1+\frac{1}{\sqrt{R}}} (r-1)^{\frac{\rho}{p-1}-\frac{p}{p-1}} r^{N-1}dr
\\&\stackrel{r-1=s}{\leq} C\left[\frac{1}{\ln R}\right]^\frac{p}{p-1} \int\limits_{\frac{1}{R}<s< \frac{1}{\sqrt{R}}} s^{\frac{\rho}{p-1}-\frac{p}{p-1}} ds
\\&\leq \left(\ln R\right)^{-\frac{p}{p-1}} \left(R^{-\frac{1}{2}\left(\frac{\rho-1}{p-1}\right)}-R^{-\frac{\rho-1}{p-1}}\right)
\\&\leq\left(\ln R\right)^{-\frac{p}{p-1}} R^{-\frac{1}{2}\left(\frac{\rho-1}{p-1}\right)}.
\end{split}\end{equation}
Combining \eqref{A2}, \eqref{DWQ7} and \eqref{LP10} we obtain \eqref{DWQ3}, which completes the proof.\end{proof}

\subsection{Proof of Theorems} At this stage, we will prove the main theorems in the following order.
\begin{proof}[Proof of Theorem \ref{TT1}.]\textbf{(i) The problem \eqref{P1} with the Neumann boundary condition \eqref{P2}.} In view of Definition \ref{WS1}, we have
\begin{equation}\label{PT1}\begin{split}
&\int_{Q_T}(|x|-1)^{-\rho}|u|^{p}\varphi  dx dt +\int_\Gamma f\varphi d\sigma dt
\leq \int_{Q_T} |u||\varphi_t|dx dt+\int_{Q_T} |u||\Delta\varphi| dx dt.
\end{split}\end{equation}
From H\"{o}lder's inequality, it follows that
\begin{equation*}\begin{split}
 \int_{Q_T} |u||\varphi_t|dx dt\leq\biggl(\int_{Q_T} (|x|-1)^{-\rho}|u|^p\varphi dx dt\biggr)^\frac{1}{p}\biggl(\,\,\underbrace{\int_{Q_T} (|x|-1)^\frac{\rho}{p-1}\varphi^{-\frac{1}{p-1}}|\varphi_t|^\frac{p}{p-1}dx dt}_{\mathcal{I}_1(\varphi,p)}\,\,\biggr)^\frac{p-1}{p}
\end{split}\end{equation*}
and
\begin{equation*}\label{PL4}\begin{split}
\int_{Q_T} |u||\Delta\varphi| dx dt\leq \biggl(\int_{Q_T} (|x|-1)^{-\rho}|u|^p\varphi dx dt\biggr)^\frac{1}{p}\biggl(\,\,\underbrace{\int_{Q_T} (|x|-1)^\frac{\rho}{p-1}\varphi^{-\frac{1}{p-1}}|\Delta\varphi|^\frac{p}{p-1}dx dt}_{\mathcal{I}_3(\varphi,p)}\,\,\biggr)^\frac{p-1}{p}.
\end{split}\end{equation*}
Then, using the $\varepsilon$-Young inequality with $\displaystyle\varepsilon=\frac{p}{3}$ in the previous estimates, we rewrite \eqref{PT1} as follows
\begin{equation}\label{PL5}\begin{split}
&\int_\Gamma f\varphi d\sigma dt\leq C(p)(\mathcal{I}_1(\varphi,p)+\mathcal{I}_3(\varphi,p)),
\end{split}\end{equation}where $\displaystyle C(p)=\frac{p-1}{p}\left(\frac{p}{3}\right)^{-\frac{1}{p-1}}.$
\\According to the results of Lemma \ref{POP} and Lemma \ref{LL2}, we get
\begin{equation*}\begin{split}
&C\int_{\partial B_2}f(x)d\sigma\leq C(p)\left(T^{-\frac{p}{p-1}}+\left[\left(\ln R\right)^{-\frac{2p}{p-1}}+\left(\ln R\right)^{-\frac{p}{p-1}}\right]R^{-\frac{1}{2}\left(\frac{\rho-p-1}{p-1}\right)}\right). \end{split}\end{equation*}
Hence, choosing $R=T$ and passing to the limit $T\to\infty$ taking account $\rho=p+1$, we get a contradiction with $$\int_{\partial B_2}f(x)d\sigma>0,$$which proves our assumption.
\\\textbf{(i)} \textbf{The problem \eqref{P1} with the Dirichlet boundary condition \eqref{P3}.}  Assume that $u\in L^p({Q_T})$ is a global weak solution to \eqref{P1}-\eqref{P3}. Then, using Definition \ref{WS2} and acting in the same way as in the above case, we arrive at
\begin{equation*}\label{CPL}\begin{split}
&-\int_\Gamma f\frac{\partial\psi}{\partial\nu} d\sigma dt\leq C(p)(\mathcal{J}_1(\psi,p)+\mathcal{J}_3(\psi,p)).
\end{split}\end{equation*}
\\Hence, using \eqref{S0} and Lemma \ref{LS90} on the left-hand side of the last inequality, we obtain
\begin{equation*}\label{DP2}\begin{split}
\int_\Gamma f\frac{\partial \psi}{\partial \nu}d\sigma dt&=\left(\int_0^T\left(1-\frac{t}{T}\right)^ldt\right)\left(\int_{\partial B_2}f(x)d\sigma\right)
\\&\leq-\frac{C(N)}{l+1}T\left(\int_{\partial B_2}f(x)d\sigma\right)\\&\leq-C_1T\left(\int_{\partial B_2}f(x)d\sigma\right).\end{split}\end{equation*}
\\From Lemma \ref{DWQ}, we get
\begin{equation*}\label{CPL2}\begin{split}
&C_1\int_{\partial B_2}f(x)d\sigma\leq C\left(T^{-\frac{p}{p-1}}+\left(\ln R\right)^{-\frac{2p}{p-1}}R^{-\frac{1}{2}\left(\frac{\rho-p-1}{p-1}\right)}+\left(\ln R\right)^{-\frac{p}{p-1}} R^{-\frac{1}{2}\left(\frac{\rho-1}{p-1}\right)}\right). \end{split}\end{equation*}
Next,  changing $R=T$ and passing $T\to\infty$ noting that $\rho=p+1$, we get a contradiction with $$\int_{\partial B_2}f(x)d\sigma>0.$$
\\\textbf{(ii) The problem \eqref{H1}  with the Neumann boundary condition \eqref{P2}}. We proceed as in the proof of part $\textbf{(i)}.$ Namely, suppose that $u\in L^p_{\text{loc}}({Q_T})$ is a global weak solution to problem \eqref{H1} under the boundary condition \eqref{P2}. Then, in view of Definition \ref{WS3} and the $\varepsilon$-Young inequality with $\displaystyle\varepsilon=\frac{p}{3}$, it follows that
\begin{equation}\label{HNC}\begin{split}
&\int_\Gamma f\varphi d\sigma dt\leq C(p)(\mathcal{I}_2(\varphi,p)+\mathcal{I}_3(\varphi,p)),
\end{split}\end{equation}where $\displaystyle C(p)=\frac{p-1}{p}\left(\frac{p}{3}\right)^{-\frac{1}{p-1}}.$
\\From Lemma \ref{POP} and Lemma \ref{LL2} we have
\begin{equation*}\begin{split}
&C\int_{\partial B_2}f(x)d\sigma\leq C(p)\left(T^{-\frac{2p}{p-1}}+\left[\left(\ln R\right)^{-\frac{2p}{p-1}}+\left(\ln R\right)^{-\frac{p}{p-1}}\right]R^{-\frac{1}{2}\left(\frac{\rho-p-1}{p-1}\right)}\right).
\end{split}\end{equation*} Choosing $R=T$ and passing to the limit $T\to\infty$ taking account $\rho=p+1$, we get a contradiction with $$\int_{\partial B_2}f(x)d\sigma>0.$$
\\\textbf{(ii) The problem \eqref{H1} with the Dirichlet boundary condition \eqref{P3}}. At this stage, noting Definition \ref{WS4}, Lemma \ref{POP} and Lemma \ref{LS90} and applying the same argument as in the proof of the assertion (i), we obtain\begin{equation}\label{HDC}\begin{split}
&C_1\int_{\partial B_2}f(x)d\sigma \leq C\left(T^{-\frac{2p}{p-1}}+\left(\ln R\right)^{-\frac{2p}{p-1}}R^{-\frac{1}{2}\left(\frac{\rho-p-1}{p-1}\right)}+\left(\ln R\right)^{-\frac{p}{p-1}} R^{-\frac{1}{2}\left(\frac{\rho-1}{p-1}\right)}\right),
\end{split}\end{equation}where $\displaystyle C>0.$
\\Choosing $R=T$ and passing to the limit $T\to\infty$ in view of $\rho=p+1$, we get a contradiction with $$\int_{\partial B_2}f(x)d\sigma>0,$$which completes the proof.\end{proof}

\begin{proof}[Proof of Theorem \ref{TT2}.]
\textbf{(i) The problem \eqref{SP1}  with the Neumann boundary condition \eqref{SP3}}. Assume that $(u,v)\in L^p_{\text{loc}}({Q_T})\times L^q_{\text{loc}}({Q_T})$ is a global weak solution to problem \eqref{SP1}-\eqref{SP3}. Due to Definition \ref{WS5} we have
\begin{equation}\label{PG01}\begin{split}
&\int_{Q_T}(|x|-1)^{-\rho}|u|^{p}\varphi  dx dt+\int_\Gamma g\varphi d\sigma dt\leq \int_{Q_T} |v||\varphi_{t}|dxdt+\int_{Q_T} |v||\Delta\varphi| dx dt,
\end{split}\end{equation}
\begin{equation}\label{PG02}\begin{split}
&\int_{Q_T}(|x|-1)^{-\rho}|v|^{q}\varphi dx dt+\int_\Gamma f\varphi d\sigma dt\leq \int_{Q_T} |u||\varphi_{t}|dxdt+\int_{Q_T} |u||\Delta\varphi| dx dt.
\end{split}\end{equation}
By H\"{o}lder's inequality, we deduce that
\begin{equation}\label{PW3}\begin{split}
 \int_{Q_T} |v||\varphi_{t}|dxdt\leq\left(\int_{Q_T}(|x|-1)^{-\rho}|v|^{q}\varphi  dx dt\right)^\frac{1}{q}\left(\mathcal{I}_1(\varphi,q)\right)^\frac{q-1}{q}
\end{split}\end{equation}
and
\begin{equation}\label{PW4}\begin{split}
\int_{Q_T} |v||\Delta\varphi| dx dt\leq\left(\int_{Q_T}(|x|-1)^{-\rho}|v|^{q}\varphi  dx dt\right)^\frac{1}{q}\left(\mathcal{I}_3(\varphi,q)\right)^\frac{q-1}{q}.
\end{split}\end{equation}
Therefore, we obtain
\begin{equation}\label{PW1}\begin{split}
 \int_{Q_T} |u||\varphi_{t}|dxdt\leq\left(\int_{Q_T}(|x|-1)^{-\rho}|u|^{p}\varphi  dx dt\right)^\frac{1}{p}\left(\mathcal{I}_1(\varphi,p)\right)^\frac{p-1}{p}
\end{split}\end{equation}
and
\begin{equation}\label{PW2}\begin{split}
\int_{Q_T} |u||\Delta\varphi| dx dt\leq\left(\int_{Q_T}(|x|-1)^{-\rho}|u|^{p}\varphi  dx dt\right)^\frac{1}{p}\left(\mathcal{I}_3(\varphi,p)\right)^\frac{p-1}{p}.
\end{split}\end{equation}
Consequently, by \eqref{PW3} and \eqref{PW4} one obtains
\begin{equation}\label{ss1}\begin{split}
\int_{Q_T}(|x|-1)^{-\rho}|u|^{p}\varphi  dx dt+\int_\Gamma g\varphi d\sigma dt&\leq \left(\int_{Q_T}(|x|-1)^{-\rho}|v|^{q}\varphi  dx dt\right)^\frac{1}{q}\\&\times\left[\left(\mathcal{I}_1(\varphi,q)\right)^\frac{q-1}{q}+\left(\mathcal{I}_3(\varphi,q)\right)^\frac{q-1}{q}\right]
\end{split}\end{equation}and from \eqref{PW1} and \eqref{PW2} it follows that
\begin{equation}\label{ss2}\begin{split}
\int_{Q_T}(|x|-1)^{-\rho}|v|^{q}\varphi dx dt+\int_\Gamma f\varphi d\sigma dt&\leq \left(\int_{Q_T}(|x|-1)^{-\rho}|u|^{p}\varphi  dx dt\right)^\frac{1}{p}\\&\times\left[\left(\mathcal{I}_1(\varphi,p)\right)^\frac{p-1}{p}+\left(\mathcal{I}_3(\varphi,p)\right)^\frac{p-1}{p}\right].
\end{split}\end{equation}
In view of \eqref{PG01}, \eqref{ss1} and \eqref{ss2}, we arrive at
\begin{equation}\label{ss3}\begin{split}
&\int_{Q_T}(|x|-1)^{-\rho}|u|^{p}\varphi  dx dt+\int_\Gamma g\varphi d\sigma dt\\&\leq \left(\int_{Q_T}(|x|-1)^{-\rho}|u|^{p}\varphi  dx dt\right)^\frac{1}{pq}
\left[\left(\mathcal{I}_1(\varphi,p)\right)^\frac{p-1}{p}+\left(\mathcal{I}_3(\varphi,p)\right)^\frac{p-1}{p}\right]^\frac{1}{q}\\&\times\left[\left(\mathcal{I}_1(\varphi,q)\right)^\frac{q-1}{q}+\left(\mathcal{I}_3(\varphi,q)\right)^\frac{q-1}{q}\right].
\end{split}\end{equation}
Similarly, from \eqref{PG02}, \eqref{ss1} and \eqref{ss2}, we have
\begin{equation}\label{ss4}\begin{split}
&\int_{Q_T}(|x|-1)^{-\rho}|v|^{q}\varphi  dx dt+\int_\Gamma f\varphi d\sigma dt\leq \left(\int_{Q_T}(|x|-1)^{-\rho}|v|^{q}\varphi  dx dt\right)^\frac{1}{pq}
\\&\times\left[\left(\mathcal{I}_1(\varphi,q)\right)^\frac{q-1}{q}+\left(\mathcal{I}_3(\varphi,q)\right)^\frac{q-1}{q}\right]^\frac{1}{p}\left[\left(\mathcal{I}_1(\varphi,p)\right)^\frac{p-1}{p}+\left(\mathcal{I}_3(\varphi,p)\right)^\frac{p-1}{p}\right].
\end{split}\end{equation}
Using the $\varepsilon$-Young inequality
$$XY\leq \frac{\varepsilon}{k} X^p+\frac{1}{k'\varepsilon^{k'-1}}Y^{k'},\,\, \frac{1}{k}+\frac{1}{k'}=1,\,\,\, X,Y,\varepsilon>0,$$
in the right-side of \eqref{ss3}, \eqref{ss4} with $k=pq, \displaystyle k'=\frac{pq}{pq-1}$ and $\displaystyle\varepsilon=\frac{pq}{2}$, we obtain the following estimates
\begin{equation*}\begin{split}
\int_\Gamma g\varphi d\sigma dt\leq C(pq)
\left[\mathbf{A}(\varphi,p)\right]^\frac{p}{pq-1}\left[\mathbf{B}(\varphi,q)\right]^\frac{pq}{pq-1}
\end{split}\end{equation*} and
\begin{equation*}\begin{split}
\int_\Gamma f\varphi d\sigma dt\leq C(pq)
\left[\mathbf{B}(\varphi,q)\right]^\frac{q}{pq-1}\left[\mathbf{A}(\varphi,p)\right]^\frac{pq}{pq-1},
\end{split}\end{equation*}where
\begin{equation*}\begin{split}
\mathbf{A}(\varphi,p):=\left(\mathcal{I}_1(\varphi,p)\right)^\frac{p-1}{p}+\left(\mathcal{I}_3(\varphi,p)\right)^\frac{p-1}{p},\end{split}\end{equation*}
\begin{equation*}\begin{split}
\mathbf{B}(\varphi,q):=\left(\mathcal{I}_1(\varphi,q)\right)^\frac{q-1}{q}+\left(\mathcal{I}_3(\varphi,q)\right)^\frac{q-1}{q}.\end{split}\end{equation*}
In view of Lemma \ref{POP}, we have
\begin{equation}\label{E1}\begin{split}
\mathbf{A}(\varphi,p)\leq C \left(T^{-\frac{1}{p}}+\mathcal{C}(\ln R)T^\frac{p-1}{p}R^{-\frac{1}{2}\left(\frac{\rho-p-1}{p}\right)}\right)\end{split}\end{equation}
and
\begin{equation}\label{E2}\begin{split}
\mathbf{B}(\varphi,q)\leq C\left(T^{-\frac{1}{q}}+\mathcal{C}(\ln R)T^\frac{q-1}{q}R^{-\frac{1}{2}\left(\frac{\rho-q-1}{q}\right)}\right), \end{split}\end{equation}where $\mathcal{C}(\ln R):=\left(\ln R\right)^{-2}+\left(\ln R\right)^{-1}.$
\\Therefore, from \eqref{E1} and \eqref{E2} it follows that
\begin{equation*}\label{E3}\begin{split}
\mathbf{A}(\varphi,p)\left[\mathbf{B}(\varphi,q)\right]^q
&\leq C \left(T^{-\frac{1}{p}}+\mathcal{C}(\ln R)T^\frac{p-1}{p}R^{-\frac{1}{2}\left(\frac{\rho-p-1}{p}\right)}\right) \left(T^{-\frac{1}{q}}+\mathcal{C}(\ln R)T^\frac{q-1}{q}R^{-\frac{1}{2}\left(\frac{\rho-q-1}{q}\right)}\right)^q
\\&= C \left(T^{-\frac{1}{p}}+\mathcal{C}(\ln R)T^\frac{p-1}{p}R^{-\frac{1}{2}\left(\frac{\rho-p-1}{p}\right)}\right) \left(T^{-1}+\mathcal{C}^q(\ln R)T^{q-1}R^{-\frac{1}{2}(\rho-q-1)}\right)
\\&= C \biggl(T^{-\frac{1}{p}-1}+\mathcal{C}^q(\ln R)T^{-\frac{1}{p}+q-1}R^{-\frac{1}{2}(\rho-q-1)}
+\mathcal{C}(\ln R)T^{\frac{p-1}{p}-1}R^{-\frac{1}{2}\left(\frac{\rho-p-1}{p}\right)}
\\&+\mathcal{C}^{q+1}(\ln R)T^{\frac{p-1}{p}-1+q}R^{-\frac{1}{2}(\frac{\rho-p-1}{p}+\rho-q-1)}\biggr).\end{split}\end{equation*}
\\Similarly, one obtains
\begin{equation*}\begin{split}
\mathbf{B}(\varphi,q)\left[\mathbf{A}(\varphi,p)\right]^p
&\leq C \left(T^{-\frac{1}{q}}+\mathcal{C}(\ln R)T^\frac{q-1}{q}R^{-\frac{1}{2}\left(\frac{\rho-q-1}{q}\right)}\right) \left(T^{-\frac{1}{p}}+\mathcal{C}(\ln R)T^\frac{p-1}{p}R^{-\frac{1}{2}\left(\frac{\rho-p-1}{p}\right)}\right)^p
\\&= C \left(T^{-\frac{1}{q}}+\mathcal{C}(\ln R)T^\frac{q-1}{q}R^{-\frac{1}{2}\left(\frac{\rho-q-1}{q}\right)}\right) \left(T^{-1}+\mathcal{C}^p(\ln R)T^{p-1} R^{-\frac{1}{2}\left(\rho-p-1\right)}\right)
\\&= C\biggl(T^{-\frac{1}{q}-1}+\mathcal{C}^p(\ln R)T^{-\frac{1}{q}+p-1}R^{-\frac{1}{2}(\rho-p-1)}+\mathcal{C}(\ln R)T^{\frac{q-1}{q}-1}R^{-\frac{1}{2}\left(\frac{\rho-q-1}{q}\right)}
\\&+\mathcal{C}^{p+1}(\ln R)T^{\frac{q-1}{q}-1+p}R^{-\frac{1}{2}(\frac{\rho-q-1}{q}+\rho-p-1)}\biggr).\end{split}\end{equation*}
Combining Lemma \ref{LL2} with the last estimates, we deduce that
\begin{equation}\label{E5}\begin{split}
\left[\int_{\partial B_2}g(x)d\sigma\right]^\frac{pq-1}{p}
&\leq C(pq) \biggl(T^{-1-q}+\mathcal{C}^q(\ln R)T^{-1}R^{-\frac{1}{2}(\rho-q-1)}
\\&+\mathcal{C}(\ln R)T^{-q}R^{-\frac{1}{2}\left(\frac{\rho-p-1}{p}\right)}\\&+\mathcal{C}^{q+1}(\ln R)R^{-\frac{1}{2}\left(\frac{\rho-p-1}{p}+\rho-q-1\right)}\biggr)\end{split}\end{equation}
and
\begin{equation}\label{E6}\begin{split}
\left[\int_{\partial B_2}f(x)d\sigma\right]^\frac{pq-1}{q}
&\leq C(pq) \biggl(T^{-1-p}+\mathcal{C}^p(\ln R)T^{-1}R^{-\frac{1}{2}(\rho-p-1)}
\\&+\mathcal{C}(\ln R)T^{-p}R^{-\frac{1}{2}\left(\frac{\rho-q-1}{q}\right)}\\&+\mathcal{C}^{p+1}(\ln R)R^{-\frac{1}{2}\left(\frac{\rho-q-1}{q}+\rho-p-1\right)}\biggr).\end{split}\end{equation}

At this stage, we have to consider two cases:
\\$\bullet$ The case $p\geq q$. Taking $R=T^\lambda, \lambda>0$ in \eqref{E5}, we have
\begin{equation}\label{+1}\begin{split}
\left[\int_{\partial B_2}g(x)d\sigma\right]^\frac{pq-1}{p}
&\leq C \biggl(T^{\theta_1}+T^{\theta_2}\mathcal{C}^{q}(\ln T^\lambda)+T^{\theta_3}\mathcal{C}(\ln T^\lambda)
+T^{\theta_4}\mathcal{C}^{q+1}(\ln T^\lambda)\biggr),\end{split}\end{equation}
where
\begin{equation*}\left\{\begin{array}{l}
\theta_1:=-1-q,\\{}\\
\theta_2:=-1+\frac{\lambda}{2}\left(-\rho+q+1\right),\\{}\\
\theta_3:=-q+\frac{\lambda}{2}\left(\frac{-\rho+p+1}{p}\right),\\{}\\
\theta_4:=\frac{\lambda}{2}\left(\frac{-\rho+p+1}{p}-\rho+q+1\right).\\{}\\ \end{array}\right.\end{equation*}
Recalling the identity \eqref{ssT} for $p\geq q$, we deduce that
\begin{equation*}\label{F1}
\theta_4<0\,\,\,\text{for all}\,\,\,\lambda>0.
\end{equation*}
Therefore, taking 
\begin{equation*}
\lambda(-\rho+p+1)<2,
\end{equation*}
we obtain
\begin{equation*}
\theta_2\leq-1+\frac{\lambda}{2}\left(-\rho+p+1\right)<0.
\end{equation*}
Similarly, it holds true
\begin{equation*}
p\theta_3:=-pq+\frac{\lambda}{2}\left(-\rho+p+1\right)<0.
\end{equation*}
Finally, passing to the limit as $T\to\infty$ in \eqref{+1}, we get a contradiction with \eqref{ss}.
\\$\bullet$ The case $p<q$. Taking $R=T^\lambda, \lambda>0$ in \eqref{E6}, it follows that
\begin{equation}\label{+2}\begin{split}
\left[\int_{\partial B_2}f(x)d\sigma\right]^\frac{pq-1}{q}
&\leq C \biggl(T^{\mu_1}+T^{\mu_2}\mathcal{C}^{p}(\ln T^\lambda)+T^{\mu_3}\mathcal{C}(\ln T^\lambda)
+T^{\mu_4}\mathcal{C}^{p+1}(\ln T^\lambda)\biggr),\end{split}\end{equation}
where
\begin{equation*}\left\{\begin{array}{l}
\mu_1:=-1-p,\\{}\\
\mu_2:=-1+\frac{\lambda}{2}\left(-\rho+p+1\right),\\{}\\
\mu_3:=-p+\frac{\lambda}{2}\left(\frac{-\rho+q+1}{q}\right),\\{}\\
\mu_4:=\frac{\lambda}{2}\left(\frac{-\rho+q+1}{q}-\rho+p+1\right).\\{}\\ \end{array}\right.\end{equation*}
From \eqref{ssT} for $p<q$, we get
\begin{equation*}
\mu_4<0\,\,\,\text{for all}\,\,\,\lambda>0.
\end{equation*}
Next, taking
\begin{equation*}
\lambda(-\rho+q+1)<2,
\end{equation*}
we conclude that
\begin{equation*}
\mu_2\leq-1+\frac{\lambda}{2}\left(-\rho+q+1\right)<0.
\end{equation*}
Moreover, we can verify the following inequality holds
\begin{equation*}
q\mu_3:=-pq+\frac{\lambda}{2}\left(-\rho+q+1\right)<0.
\end{equation*}
Hence, passing to the limit as $T\to\infty$ in \eqref{+2}, we get a contradiction with \eqref{ss}.
\\\textbf{(i) The problem \eqref{SP1}  with the Dirichlet boundary condition \eqref{SP4}}.
 We proceed as in the proof of part $\textbf{(i)}.$ Conequently, from Definition \ref{WS8} with the H\"{o}lder and $\varepsilon$-Young inequality, we obtain
\begin{equation*}\begin{split}
-\int_\Gamma g\frac{\partial\psi}{\partial\nu} d\sigma dt\leq C(pq)
\left[\mathbf{C}(\psi,p)\right]^\frac{p}{pq-1}\left[\mathbf{D}(\psi,q)\right]^\frac{pq}{pq-1}
\end{split}\end{equation*} and
\begin{equation*}\begin{split}
-\int_\Gamma f\frac{\partial\psi}{\partial\nu} d\sigma dt\leq C(pq)
\left[\mathbf{D}(\psi,q)\right]^\frac{q}{pq-1}\left[\mathbf{C}(\psi,p)\right]^\frac{pq}{pq-1},
\end{split}\end{equation*}where
\begin{equation*}\begin{split}
\mathbf{C}(\psi,p):=\left(\mathcal{J}_1(\psi,p)\right)^\frac{p-1}{p}+\left(\mathcal{J}_3(\psi,p)\right)^\frac{p-1}{p},\end{split}\end{equation*}
\begin{equation*}\begin{split}
\mathbf{D}(\psi,q):=\left(\mathcal{J}_1(\psi,q)\right)^\frac{q-1}{q}+\left(\mathcal{J}_3(\psi,q)\right)^\frac{q-1}{q}.\end{split}\end{equation*}
From Lemma \ref{DWQ}, we deduce that
\begin{equation}\label{EA3}\begin{split}
\mathbf{C}(\psi,p)\leq C\left(T^{-\frac{1}{p}}+T^\frac{p-1}{p}\left[\left(\ln R\right)^{-2}R^{-\frac{1}{2}\left(\frac{\rho-p-1}{p}\right)}+\left(\ln R\right)^{-1} R^{-\frac{1}{2}\left(\frac{\rho-1}{p}\right)}\right]\right)\end{split}\end{equation}
and
\begin{equation}\label{EA4}\begin{split}
\mathbf{D}(\psi,q)\leq C\left(T^{-\frac{1}{q}}+T^\frac{q-1}{q}\left[\left(\ln R\right)^{-2}R^{-\frac{1}{2}\left(\frac{\rho-q-1}{q}\right)}+\left(\ln R\right)^{-1} R^{-\frac{1}{2}\left(\frac{\rho-1}{q}\right)}\right]\right). \end{split}\end{equation}
Then, due to the last estimates, we obtain 
\begin{equation*}\label{E30}\begin{split}
\left[\mathbf{C}(\psi,p)\right]\left[\mathbf{D}(\psi,p)\right]^q
&\leq C\left(T^{-\frac{1}{p}}+T^\frac{p-1}{p}\left[\left(\ln R\right)^{-2}R^{-\frac{1}{2}\left(\frac{\rho-p-1}{p}\right)}+\left(\ln R\right)^{-1} R^{-\frac{1}{2}\left(\frac{\rho-1}{p}\right)}\right]\right)
\\&\times\left(T^{-1}+T^{q-1}\left[\left(\ln R\right)^{-2q}R^{-\frac{1}{2}\left({\rho-q-1}\right)}+\left(\ln R\right)^{-q} R^{-\frac{1}{2}\left({\rho-1}\right)}\right]\right)
\\&= C \biggl(T^{-\frac{1}{p}-1}+T^{-\frac{1}{p}+q-1}\left[\left(\ln R\right)^{-2q}R^{-\frac{1}{2}\left({\rho-q-1}\right)}+\left(\ln R\right)^{-q} R^{-\frac{1}{2}\left({\rho-1}\right)}\right]
\\&+T^{\frac{p-1}{p}-1}\left[\left(\ln R\right)^{-2}R^{-\frac{1}{2}\left(\frac{\rho-p-1}{p}\right)}+\left(\ln R\right)^{-1} R^{-\frac{1}{2}\left(\frac{\rho-1}{p}\right)}\right]
\\&+T^{\frac{p-1}{p}+q-1}\biggl[\left(\ln R\right)^{-2-2q}R^{-\frac{1}{2}\left(\frac{\rho-p-1}{p}+{\rho-q-1}\right)}+\left(\ln R\right)^{-2-q}R^{-\frac{1}{2}\left(\frac{\rho-p-1}{p}+{\rho-1}\right)}
 \\&+\left(\ln R\right)^{-1-2q} R^{-\frac{1}{2}\left(\frac{\rho-1}{p}+{\rho-q-1}\right)}+\left(\ln R\right)^{-1-q} R^{-\frac{1}{2}\left(\frac{\rho-1}{p}+{\rho-1}\right)}\biggr]\biggr).\end{split}\end{equation*}
Similarly,
\begin{equation*}\begin{split}
\left[\mathbf{D}(\psi,p)\right]\left[\mathbf{C}(\psi,p)\right]^p
&\leq C\left(T^{-\frac{1}{q}}+T^\frac{q-1}{q}\left[\left(\ln R\right)^{-2}R^{-\frac{1}{2}\left(\frac{\rho-q-1}{q}\right)}+\left(\ln R\right)^{-1} R^{-\frac{1}{2}\left(\frac{\rho-1}{q}\right)}\right]\right)
\\&\times\left(T^{-1}+T^{p-1}\left[\left(\ln R\right)^{-2p}R^{-\frac{1}{2}\left({\rho-p-1}\right)}+\left(\ln R\right)^{-p} R^{-\frac{1}{2}\left({\rho-1}\right)}\right]\right)
\\&= C \biggl(T^{-\frac{1}{q}-1}+T^{-\frac{1}{q}+p-1}\left[\left(\ln R\right)^{-2p}R^{-\frac{1}{2}\left({\rho-p-1}\right)}+\left(\ln R\right)^{-p} R^{-\frac{1}{2}\left({\rho-1}\right)}\right]
\\&+T^{\frac{q-1}{q}-1}\left[\left(\ln R\right)^{-2}R^{-\frac{1}{2}\left(\frac{\rho-q-1}{q}\right)}+\left(\ln R\right)^{-1} R^{-\frac{1}{2}\left(\frac{\rho-1}{q}\right)}\right]
\\&+T^{\frac{q-1}{q}+p-1}\biggl[\left(\ln R\right)^{-2-2p}R^{-\frac{1}{2}\left(\frac{\rho-q-1}{q}+{\rho-p-1}\right)}+\left(\ln R\right)^{-2-p}R^{-\frac{1}{2}\left(\frac{\rho-q-1}{q}+{\rho-1}\right)}
 \\&+\left(\ln R\right)^{-1-2p} R^{-\frac{1}{2}\left(\frac{\rho-1}{q}+{\rho-p-1}\right)}+\left(\ln R\right)^{-1-p} R^{-\frac{1}{2}\left(\frac{\rho-1}{q}+{\rho-1}\right)}\biggr]\biggr).\end{split}\end{equation*}

In view of Lemma \ref{LS90} and the last estimates, we arrive at
\begin{equation}\label{E8}\begin{split}
\left[\int_{\partial B_2}g(x)d\sigma\right]^\frac{pq-1}{p}
&\leq C\biggl(T^{-1-q}+T^{-1}\left[\left(\ln R\right)^{-2q}R^{-\frac{1}{2}\left({\rho-q-1}\right)}+\left(\ln R\right)^{-q} R^{-\frac{1}{2}\left({\rho-1}\right)}\right]
\\&+T^{-q}\left[\left(\ln R\right)^{-2}R^{-\frac{1}{2}\left(\frac{\rho-p-1}{p}\right)}+\left(\ln R\right)^{-1} R^{-\frac{1}{2}\left(\frac{\rho-1}{p}\right)}\right]
\\&+\left(\ln R\right)^{-2-2q}R^{-\frac{1}{2}\left(\frac{\rho-p-1}{p}+{\rho-q-1}\right)}+\left(\ln R\right)^{-2-q}R^{-\frac{1}{2}\left(\frac{\rho-p-1}{p}+{\rho-1}\right)}
 \\&+\left(\ln R\right)^{-1-2q} R^{-\frac{1}{2}\left(\frac{\rho-1}{p}+{\rho-q-1}\right)}+\left(\ln R\right)^{-1-q} R^{-\frac{1}{2}\left(\frac{\rho-1}{p}+{\rho-1}\right)}\biggr)\end{split}\end{equation}
and
\begin{equation}\label{E9}\begin{split}
\left[\int_{\partial B_2}f(x)d\sigma\right]^\frac{pq-1}{q}
&\leq C\biggl(T^{-1-p}+T^{-1}\left[\left(\ln R\right)^{-2p}R^{-\frac{1}{2}\left({\rho-p-1}\right)}+\left(\ln R\right)^{-p} R^{-\frac{1}{2}\left({\rho-1}\right)}\right]
\\&+T^{-p}\left[\left(\ln R\right)^{-2}R^{-\frac{1}{2}\left(\frac{\rho-q-1}{q}\right)}+\left(\ln R\right)^{-1} R^{-\frac{1}{2}\left(\frac{\rho-1}{q}\right)}\right]
\\&+\left(\ln R\right)^{-2-2p}R^{-\frac{1}{2}\left(\frac{\rho-q-1}{q}+{\rho-p-1}\right)}+\left(\ln R\right)^{-2-p}R^{-\frac{1}{2}\left(\frac{\rho-q-1}{q}+{\rho-1}\right)}
 \\&+\left(\ln R\right)^{-1-2p} R^{-\frac{1}{2}\left(\frac{\rho-1}{q}+{\rho-p-1}\right)}+\left(\ln R\right)^{-1-p} R^{-\frac{1}{2}\left(\frac{\rho-1}{q}+{\rho-1}\right)}\biggr).\end{split}\end{equation}
Next, we have to discuss two cases.
\\$\bullet$ The case $p\geq q$. Choosing $R=T^\lambda, \lambda>0$ in \eqref{E8}, we obtain
\begin{equation*}\begin{split}
\left[\int_{\partial B_2}g(x)d\sigma\right]^\frac{pq-1}{p}
&\leq C \biggl(T^{\eta_1}+T^{\eta_2}\left(\ln T^\lambda\right)^{-2q}+T^{\eta_3}\left(\ln T^\lambda\right)^{-q} 
\\&+T^{\eta_4}\left(\ln T^\lambda\right)^{-2}+T^{\eta_5}\left(\ln T^\lambda\right)^{-1}
\\&+T^{\eta_6}\left(\ln T^\lambda\right)^{-2-2q}+T^{\eta_7}\left(\ln T^\lambda\right)^{-2-q}
\\&+T^{\eta_8}\left(\ln T^\lambda\right)^{-1-2q}+T^{\eta_9}\left(\ln T^\lambda\right)^{-1-q} \biggr),\end{split}\end{equation*}
where
\begin{equation*}\left\{\begin{array}{l}
\eta_1:=-(1+q),\,\,\,\,\,\,\,\, \eta_2:=-1-\frac{\lambda}{2}(\rho-q-1),\,\,\,\ \eta_3:=-1-\frac{\lambda}{2}(\rho-1), \\{}\\
\eta_4:=-q-\frac{\lambda}{2}\left(\frac{\rho-p-1}{p}\right),\,\,\,\eta_5:=-q-\frac{\lambda}{2}\left(\frac{\rho-1}{p}\right),\\{}\\
\eta_6:=-\frac{\lambda}{2}\left(\frac{\rho-p-1}{p}+{\rho-q-1}\right),\,\,\,\, \eta_7:=-\frac{\lambda}{2}\left(\frac{\rho-p-1}{p}+{\rho-1}\right),\\{}\\ \eta_8:=-\frac{\lambda}{2}\left(\frac{\rho-1}{p}+{\rho-q-1}\right),\,\,\,\,\,\eta_9:=-\frac{\lambda}{2}\left(\frac{\rho-1}{p}+{\rho-1}\right).
\end{array}\right.\end{equation*}
In view of \eqref{ssT}, we deduce that
$$\eta_1, \eta_3, \eta_5, \eta_6, \eta_7, \eta_8, \eta_9<0 \,\,\,\text{for all}\,\,\,\lambda>0.$$
Next, taking
\begin{equation*}
\lambda(-\rho+q+1)<2,
\end{equation*}
and noting that $p\geq q$, we obtain
\begin{equation*}
\eta_2\leq-1+\frac{\lambda}{2}\left(-\rho+q+1\right)<0
\end{equation*}and
\begin{equation*}
p\eta_4\leq-pq+\frac{\lambda}{2}\left(-\rho+q+1\right)<0.
\end{equation*}
Then, passing to the limit as $T\to\infty$, we get a contradiction with \eqref{ss}.
\\$\bullet$ The case $p<q$. Moreover, taking $R=T^\lambda, \lambda>0$ in \eqref{E9}, we have
\begin{equation*}\begin{split}
\left[\int_{\partial B_2}f(x)d\sigma\right]^\frac{pq-1}{q}
&\leq C \biggl(T^{\nu_1}+T^{\nu_2}\left(\ln T^\lambda\right)^{-2p}+T^{\nu_3}\left(\ln T^\lambda\right)^{-p}
\\&+T^{\nu_4}\left(\ln T^\lambda\right)^{-2}+T^{\nu_5}\left(\ln T^\lambda\right)^{-1}
\\&+T^{\nu_6}\left(\ln T^\lambda\right)^{-2-2p}+T^{\nu_7}\left(\ln T^\lambda\right)^{-2-p}
\\&+T^{\nu_8}\left(\ln T^\lambda\right)^{-1-2p}+T^{\nu_9}\left(\ln T^\lambda\right)^{-1-p} \biggr),\end{split}\end{equation*}
where
\begin{equation*}\left\{\begin{array}{l}
\nu_1:=-(1+p),\qquad\qquad\qquad\,\, \nu_2:=-1-\frac{\lambda}{2}(\rho-p-1),\\{}\\ \nu_3:=-1-\frac{\lambda}{2}(\rho-1),\qquad\qquad \nu_4:=-p-\frac{\lambda}{2}\left(\frac{\rho-q-1}{q}\right),\\{}\\ \nu_5:=-p-\frac{\lambda}{2}\left(\frac{\rho-1}{q}\right),\qquad\qquad\nu_6:=-\frac{\lambda}{2}\left(\frac{\rho-q-1}{q}+{\rho-p-1}\right),\\{}\\ \nu_7:=-\frac{\lambda}{2}\left(\frac{\rho-q-1}{q}+{\rho-1}\right)
,\quad\,\nu_8:=-\frac{\lambda}{2}\left(\frac{\rho-1}{q}+{\rho-p-1}\right),\\{}\\ \nu_9:=-\frac{\lambda}{2}\left(\frac{\rho-1}{q}+{\rho-1}\right).
\end{array}\right.\end{equation*}
Finally, acting in the same way as in the previous case, we get a contradiction with \eqref{ss}, which completes the proof.

\textbf{(II).} The proof of the second part will be similar to the first case, just instead of $$\mathcal{I}_1(\varphi,p),\mathcal{I}_1(\varphi,q)$$ and $$\mathcal{J}_1(\psi,p), \mathcal{J}_1(\psi,q)$$ will replace by $$\mathcal{I}_2(\varphi,p),\mathcal{I}_2(\varphi,q)$$ and $$\mathcal{J}_2(\psi,p), \mathcal{J}_2(\psi,q),$$ respectively. Thus, we simply omit it.
\end{proof}

\section*{\large Acknowledgments}
This research has been funded by the Science Committee of the Ministry of Education and Science of the Republic of Kazakhstan (AP14972726). BT is also supported by the FWO Odysseus 1 grant G.0H94.18N: Analysis and Partial Differential Equations and by the Methusalem programme of the Ghent University Special Research Fund (BOF) (Grant number 01M01021).

\end{document}